\documentclass[12pt, reqno]{amsart}
\usepackage{indentfirst,amssymb,amsmath,amsthm,enumerate, hyperref}
\newtheorem{theorem}{Theorem}[section]
\newtheorem{example}[theorem]{Example}

\newtheorem{lemma}{Lemma}[section]

\renewcommand{\epsilon}{\varepsilon}
\newcommand{\be}{\begin{equation}}
\newcommand{\ee}{\end{equation}}
\newcommand{\beas}{\begin{eqnarray*}}
\newcommand{\eeas}{\end{eqnarray*}}
\newcommand{\bea}{\begin{eqnarray}}
\newcommand{\eea}{\end{eqnarray}}

\numberwithin{equation}{section}
\begin{document}
\title[Exploring Topological Transitivity in Families of Functions]{Exploring Topological Transitivity in Families of Functions}
\author[A. Singh]{Anil Singh}
\date{}
\address{Anil Singh, Department of Mathematics, Maulana Azad Memorial College, Jammu 180006, India.}
\email{anilmanhasfeb90@gmail.com}
\author[B. Lal]{Banarsi Lal}
\address{ Banarsi Lal, Department of Mathematics, School of Sciences, Cluster University of Jammu, Jammu 180001, India.}
\email{banarsiverma644@gmail.com}

\renewcommand{\thefootnote}{\fnsymbol{footnote}}
\begin {abstract} We have established various criteria for the topological transitivity of families of continuous (holomorphic) functions. Furthermore, by leveraging the properties of expanding families of meromorphic functions, we offer an alternative proof of Montel’s three-point theorem. \end{abstract}
\maketitle
\enlargethispage{3pt}
\section{Introduction}
Let $D$ be an open set in the finite complex plane $\mathbb{C}$. We denote the class of all continuous functions from $D$ to a metric space $Y$ by $C(D, Y)$. A family of functions $\mathcal{F} \subseteq C(D, Y)$ is said to be normal if every sequence in $\mathcal{F}$ contains a subsequence that converges locally uniformly on $D$. Additionally, $\mathcal{F}$ is said to be normal at a point $z \in D$ if it is normal in some neighborhood of $z$ in $D$ (see \cite{schiff1993normal, zalcman1998normal}). The notion of normality is key in understanding the compactness properties of families of functions and play a central role in the study of dynamics of entire and meromorphic functions. If the limit function of every sequence of a normal family $\mathcal{F}$ is also in $\mathcal{F}$, then $\mathcal{F}$ is a compact family.

\medskip

In this paper, beyond normality, we explore topologically transitive families of functions. The concept of topological transitivity is well established in the study of dynamical systems. Recently, this notion has been extended to families of functions by L. Bernal-Gonzalez \cite{bernal2021universality} and T. Meyrath \cite{meyrath2022}. A family $\mathcal{F}\subseteq C(D, Y)$ is said to be {\it topologically transitive} at $z_0 \in D$ with respect to a set $B \subseteq Y$ if, for all non empty open sets $U$ in $D$ containing $z_0$ and $V \subset Y$ with $V \cap B \neq \emptyset$, there exists $f \in \mathcal{F}$ satisfying $f(U) \cap V \neq \emptyset$. Further, $\mathcal F$ is said to be {\it topologically transitive} with respect to a set $B \subseteq Y$ if, for all non empty open sets $U \subset D$ and $V \subset Y$ with $V \cap B \neq \emptyset$, there exists $f \in \mathcal{F}$ satisfying $f(U) \cap V \neq \emptyset$. If $B = Y$, then the phrase ``with respect to" is omitted indicating that the transitivity is considered in the broader context of the entire space. 

It is worth noting that if a family $\mathcal{F}$ of meromorphic functions on an open set $D$ is not topologically transitive, then, by Montel's three-point theorem (\textit{a family of meromorphic functions defined on open set $D$ in the complex plane is normal if it omits at least three distinct values in $\mathbb{C}_\infty$}), it follows that $\mathcal{F}$ is normal on some open subset contained in $D.$ The converse holds if $\mathcal F$ is locally bounded. Note that converse also holds for certain classes of unbounded families of functions, for example, if a non constant entire function $f$ has an attracting fixed point, then the family of iterates of $f$ is not topologically transitive on Fatou set of $f$.

\medskip
 
 A subset $E \subset D$ is said to be forward invariant under $\mathcal{F}$ if $f(E) \subseteq E$ for all $f \in \mathcal{F}$ and it is said to be backward invariant if $f^{-1}(E) \subseteq E$ for all $f \in \mathcal{F}$. $E$ is completely invariant under $\mathcal{F}$ if it is forward invariant and backward invariant under $\mathcal{F}.$  The existence of proper open subsets of $D$  that are forward invariant under $\mathcal{F}$, prevents  $\mathcal{F}$ from being  topologically transitive on $D$. In such cases, the transitivity can be considered with respect to such sets. In the case of a family of iterates of entire or meromorphic functions, transitivity on the Fatou set, if it exists, is considered with respect to the Fatou set itself. For more details about Fatou and Julia sets in complex dynamics, one can see \cite{beardon1991iteration, bergweiler1993iteration, carleson1996complex}.
 
\medskip
Finally, we address the concept of universal family. Family $\mathcal{F}$ is said to be {\it universal} if  there exists $z_0\in D$ such that the set $\mathcal{O}(z_0,\mathcal{F}):=\{f(z_0):f\in\mathcal{F}\}$ is dense in ${Y}.$ Element $z_0$ with this property is called a universal for the family $\mathcal{F}.$ The Universality Criterion [\cite{grosse2011linear}, Theorem 1.57] states that \textit{ a dense set of universal elements for $\mathcal{F}$ exists if and only if $\mathcal{F}$ is topologically transitive.}

\medskip

 For a given domain $D\subseteq \mathbb{C}$ and a family $\mathcal{F}$ of holomorphic (meromorphic) functions, the authors in \cite{charak2020fatou} defined $F(\mathcal{F})$, a subset of $D$ on which $\mathcal{F}$ is normal and $J(\mathcal{F}):= D\setminus F(\mathcal{F})$. If $\mathcal{F}$ consists of iterates of an entire or meromorphic function $f$ or a semigroup of entire functions, then $F(\mathcal{F})$ and $J(\mathcal{F})$ correspond to the Fatou set  and Julia set of $f,$ respectively. It is important to note  that forward invariant set $J(\mathcal{F})$ of a family of  holomorphic or meromorphic functions can not contain a universal element unless $J(\mathcal{F})=\mathbb{C}_{\infty}.$ Various other interesting properties of the sets $F(\mathcal{F})$ and $J(\mathcal{F})$ are studied in \cite{charak2021fatou, charak2020fatou}. 

 \medskip
 
The remainder of this paper is organized into four sections. In Section \ref{sec:minimalfamilies}, we introduce minimal families and establish a necessary and sufficient condition for transitivity. Section \ref{sec:limitsets}
focuses on various limit sets, where we present some intriguing results involving the interaction between normality and minimal families. In Section \ref{sec:transitivityconditions}, we explore several sufficient conditions for the transitivity of families of functions, and we also investigate conditions that allow the transfer of transitivity from one family to another. Finally, in Section \ref{sec:montelstheorem}, we provide a simple proof of Montel's theorem using the expanding property of non-normal families of meromorphic functions.

\section{Minimal Family}\label{sec:minimalfamilies}
 A family $\mathcal{F} \subseteq C(D, Y)$ is said to be \textit{minimal} on $D$ if for all $z\in D,$ $\mathcal{O}(z,\mathcal{F})=\{f(z):f\in\mathcal{F}\}$ is dense in $Y.$ 
Further, we say that family $\mathcal F$ is \textit{heriditrarily minimal} on $D$ if  every infinite subfamily $\mathcal F'$ of $\mathcal F$ is minimal.

\medskip

 If non-constant entire function $f$ has an attracting fixed point, then the family $\left\{f^n\right\}$ of iterates can not be minimal on ${F(f)}.$ Even, it can not be minimal if we consider $\left\{f^n\right\}$  as a family of functions from $F(f)$ to itself. Whereas, if $z_0$ is a non-exceptional point of $f$ in $J(f)$, the we know that the forward orbit of $z_0$ is dense in $J(f)$. The latter statement holds true for semigroups of entire functions [\cite{charak2021fatou}, Theorem 2.1].

\begin{theorem}
    Let $\mathcal{F}$ be a semigroup  of  entire functions. Then the following are equivalent.
    \begin{itemize}
    \item[(a)] $\mathcal{F}$ is    minimal.
    \item[(b)] Only closed set $K$ with $f(K)\subseteq K$  for all $f\in \mathcal{F}$ is $\emptyset$ or $\mathbb{C}.$
     \item[(c)] Only open set $G$ with $f^{-1}(G)\subseteq G$  for all $f\in \mathcal{F}$ is $\emptyset$ or $\mathbb{C}.$
    \end{itemize}
\end{theorem}
\begin{proof}
   $(a) \Rightarrow (b)$: Let $K$ be a non-empty closed set satisfying $f(K)\subseteq K$  for all $f\in \mathcal{F},$ and $z\in K.$ Then the set $\{f(z):f\in\mathcal{F}\}\subseteq K$ is dense in $\mathbb{C},$  so that $K=\mathbb{C} .$ 

    $(b) \Rightarrow (a)$: Let $z\in\mathbb{C}$ and put $A=\{f(z):f\in\mathcal{F}\}$. We need to show that $A$ is dense in $\mathbb{C}.$ Let $f_0\in \mathcal{F}$ and $w\in f_0(\overline{A}).$ Then $w=f_0(\zeta)$ for some $\zeta\in \overline{A}.$ There exists a sequence $\{f_n\}$ in $\mathcal{F}$ such that $f_n(z)\to \zeta$. Continuity of $f_0$ implies that $f_0\circ f_n(z)\to w$. Since $f_0\circ f_n\in \mathcal{F}$, $w\in \overline{A}.$  Therefore $f_0(\overline{A})\subseteq \overline{A},$ and so by our assumption, $\overline{A}=\mathbb{C}.$

    $(b) \Rightarrow (c)$: Let $G$ be an open set in $\mathbb C$ such that $f^{-1}(G) \subseteq G$ for all $f \in \mathcal F$. Then $K : = \mathbb C \setminus G$ is a closed set in $\mathbb C$ satisfying $f(K) \subseteq K$ for all $f \in \mathcal F$, so that $(b)$ implies either $K = \emptyset$ or $\mathbb C$. Thus either $G = \mathbb C$ or $G = \emptyset$.

     $(c) \Rightarrow (b)$: Follows analogously as $(b) \Rightarrow (c)$.

     $(c) \Rightarrow (a)$: Let \( z \in \mathbb{C} \) and define \( A = \{ f(z) : f \in \mathcal{F} \} \). We aim to demonstrate that \( A \) is dense in \( \mathbb{C} \). Consider the set \( G = \mathbb{C} \setminus \overline{A} \). Since \( \overline{A} \) is the closure of \( A \), \( G \) is an open set. Assume, for the sake of contradiction, that \( G \neq \emptyset \). We claim that for any \( f \in \mathcal{F} \), the preimage \( f^{-1}(G) \) is also contained in \( G \).

Let \( \zeta \in f^{-1}(G) \). This implies \( f(\zeta) \notin \overline{A} \), so there exists a neighborhood \( V \) of \( f(\zeta) \) such that \( V \cap A = \emptyset \). Consequently, \( f^{-1}(V) \) is a neighborhood of \( \zeta \) that does not intersect \( A \). If there were some \( \xi \in f^{-1}(V) \cap A \), then \( f(\xi) \) would lie in both \( V \) and \( f(A) \). Since \( \mathcal{F} \) is a semigroup, we have \( f(A) \subseteq A \), implying \( f(\xi) \in A \) as well, contradicting the choice of \( V \). Therefore, \( \zeta \) cannot belong to \( \overline{A} \), leading to \( \zeta \in G \).

Thus $(c)$ implies that $G = \mathbb C$. However, this would imply \( A \) is empty, which is a contradiction. Therefore, we conclude that \( G \) must be empty, and thus \( \overline{A} = \mathbb{C} \), confirming that \( A \) is indeed dense in \( \mathbb{C} \).

   \end{proof}
The Universality Criterion implies that a minimal family is topologically transitive. We are interested in showing the converse: if a family is topologically transitive, then it cannot be anything but minimal. This establishes a strong relationship between the two concepts, emphasizing their interconnectedness in the study of dynamical systems.
\begin{lemma}\label{Thm:closednessofuniversals}
Let $\mathcal{F}\subset C(D, Y).$  Then the set $ M\left(\mathcal{F}\right)$ of all points of $D$ at which the family $\mathcal{F}$ is universal is a closed set. 
\end{lemma}
\begin{proof} We take $Y=\mathbb{C}$,  as the proof remains unchanged except for replacing the modulus with the metric on $Y.$
Take a point $z_0 \in \overline{M(\mathcal{F})}$. Then there exists a sequence of points $\{z_n\}$ in $M(\mathcal{F})$ such that $z_n \rightarrow z_0$ as $n \rightarrow \infty$. For each $n,$ the set $\{f(z_n) : f \in \mathcal F\}$ is dense in $\mathbb C$. Suppose that $\{f(z_0) : f \in \mathcal F\}$ is not dense in $\mathbb C$. Then there exists some $w\in\mathbb{C}$ and an $\epsilon > 0$ such that 
\begin{equation}\label{closed1}
|f(z_0) - w| > \epsilon \; \mbox{ for all } \; f \in \mathcal F.
\end{equation}
Now continuity of $f$ at $z_0$ and denseness of the set $\{f(z_n) : f \in \mathcal F\}$ implies that for above $\epsilon$, we have 
\begin{equation}\label{closed2}
|f(z_n) - f(z_0)| < \frac{\epsilon}{2}\;\text{and}\;|f(z_n) - w| < \frac{\epsilon}{2}
\end{equation} for some $f$ and large value of $n.$

From (\ref{closed1}) and (\ref{closed2}), we get
\begin{align*}
|f(z_0) - w| \leq |f(z_0) - f(z_n)| + |f(z_n) - w| < \epsilon,
\end{align*}
which is a contradiction. Thus the set $\{f(z_0) : f \in \mathcal F\}$ is dense in $\mathbb C$, and therefore $z_0 \in M(\mathcal{F})$ which shows that $M(\mathcal{F})$ is closed.
\end{proof}
\begin{theorem}\label{TTM}
A family $\mathcal{F}\subseteq C(D,Y)$ is topologically transitive on $D$ if and only if $\mathcal{F}$ is minimal.
\end{theorem}
 Proof of Theorem \ref{TTM} follows from  the Universality Criterion and Lemma \ref{Thm:closednessofuniversals}. 
\section{Limit Sets }\label{sec:limitsets}
For a family $\mathcal{F}$ of holomorphic functions on a domain $D$, we denote by $\overline{\mathcal{F}}$ the set of all functions
$f:D\longmapsto\mathbb{C}$ such that there is a sequence $\{f_n\}\subseteq \mathcal{F}$ that converges locally uniformly to $f$ on $D$.
\begin{theorem}\label{cccc}
Suppose that a family $\mathcal{F}$ of holomorphic functions on a domain $D$ is minimal and normal. Then $\mathbb{C}\subseteq \overline{\mathcal{F}}.$
\end{theorem}
\begin{proof}
Let $z\in\mathbb{C}.$ Choose a countable dense set $A=\{z_n:n=1,2,\cdots\}$ in $D$. Since $\mathcal{O}(z_n,\mathcal{F})$ is dense for each $n$, there exists a sequence $\{f_{nk}\}_{k=1}^{\infty}$ such that $f_{nk}(z_n)\to z$ as $k\to\infty .$ Then the diagonal sequence $\{f_{nn}\}$ converges pointwise to $z$ on $A$. Since $\mathcal{F}$ is normal family on $D,$ $\{f_{nn}\}$ has a subsequence which converges locally uniformly on $D$ to holomorphic function $f$. But $f=z$ on a dense set $A$, hence $f\equiv z$ on $D$. This gives $z\in \overline{\mathcal{F}}.$
\end{proof}
If the family $\mathcal F$ is normal and heriditrarily minimal, then we have the following interesting result:
\begin{theorem}
Let $\mathcal F$ be a normal and heriditrarily minimal family of entire functions. Then $\overline{\mathcal F} = \mathcal H(\mathbb C)$, the set of all holomorphic functions on $\mathbb C$.
\end{theorem}
\begin{proof}
Let $f \in \mathcal H(\mathbb C)$. Without loss of generality, we assume that $f$ is non-constant, since otherwise the result follows from Theorem \ref{cccc}. We choose some countable set $\{z_j\}$ in $\mathbb C$ such that $\{z_j\}$ has a limit point in $\mathbb C$. Since $\mathcal F$ is heriditrarily minimal, there exists a sequence $\{f_{n, 1}\} \subseteq \mathcal F$ such that $f_{n, 1}(z_1) \rightarrow f(z_1)$. Again by heriditrarily minimality of $\mathcal F$, we get a subsequence $\{f_{n, 2}\}$ of $\{f_{n, 1}\}$ such that $f_{n, 2}(z_j) \rightarrow f(z_j)$ for $j = 1, 2$. Continuing this way, we get a subsequence $\{f_{n, k}\}$ of $\{f_{n, k - 1}\}$ such that $f_{n, k}(z_j) \rightarrow f(z_j)$ for $j = 1, 2, \cdots, k$. If we consider the diagonal sequence $\{f_{n, n}\}$, then clearly $f_{n, n}(z_j) \rightarrow f(z_j)$ for all $j$. Now, normality of $\mathcal F$ implies that there exists an entire function $g$ such that $f_{n, n} \rightarrow g$ locally uniformly on $\mathbb C$. By uniqueness of limits, $g(z_j) = f(z_j)$. Thus $f \equiv g$, by identity theorem, and hence the result follows.
\end{proof}

\begin{theorem}
Let $\mathcal{F}$ be a family of holomorphic functions from an open set $D$ to $\mathbb{C},$ and $\omega(z, \mathcal{F})$  be the set of limit points of $\mathcal{O}(z,\mathcal{F}).$ Then for any $z_0,$
\begin{itemize}
\item[(a)] $\omega(z_0,\mathcal{F})$ is a closed set.
\item[(b)] $\omega(z_0,\mathcal{F})$ is forward invariant, if $\mathcal{F}$ is semigroup of entire functions.
\end{itemize}
\end{theorem}
\begin{proof}
(a)  Let $w\notin \omega(z_0,\mathcal{F}).$ Then there is a neighbourhood $N$ of $w$ such that $N\cap O(z_0,\mathcal{F})\subset \{w\}.$ This implies that for any $\zeta\in N$ we must have $N\cap O(z_0,\mathcal{F})\subset \{\zeta\}.$ that is no point in $N$ is a limit point of $\omega(z_0,\mathcal{F}).$ Hence $N\cap \omega(z_0,\mathcal{F})=\emptyset$ and so $w\notin {\overline{\omega(z_0,\mathcal{F})}}.$
\medskip

(b) Choose $f_0\in\mathcal{F}$ and $w\in \omega(z_0,\mathcal{F}).$ Let $V$ be an open set containing $f_0(w).$ From continuity of $f_0,$ we can choose an open set $U$ containing $w$ such that $f_0(U)\subseteq V.$ We can assume that $U$ is bounded. Since $w$ is a limit point of $O(z_0,\mathcal{F}),$ there exists a sequence $\{f_n\}$ in $\mathcal{F}$ such that $f_n(z_0)\in U$ for all $n.$ Then $f_0\circ f_n(z_0)\in V$ for all $n.$ We claim that $f_0\circ f_n(z_0)$ is an infinite sequence. Suppose that it is a finite sequence, we may suppose if necessary, after renaming, $f_0\circ f_n(z_0)=c$ for all n. Since  $\{f_n(z_0)\}$  has a limit point, the identity theorem reduces $f_0$ to a constant, which is not true. Hence the claim, and so $f_0(w)\in \omega(z_0,\mathcal{F})$ showing that $\omega(z_0,\mathcal{F})$ is forward invariant.

\end{proof}
For an arbitrary family of functions $\omega(z_0,\mathcal{F})$ need not be forward invariant.
\begin{example}
    Consider the family $\mathcal{F}=\{z^n+1:n=1,2,\cdots\}.$ Then $\omega(1/2,\mathcal{F})=\{1\}$ and for any $f_n\in\mathcal{F}$ $f_n(1)\neq 1.$ 
\end{example}

A point $z_0$ is said to be a \textit{non-wandering} point of a family $\mathcal{F}$ of holomorphic functions on an open set $D$ if for every open neighbourhood $U$ of $z_0$, there exists a function $f \in \mathcal F$ such that $f(U) \cap U \neq \emptyset$.

By Montel's three point theorem, it is clear that $z_0 \in J(\mathcal F)$ implies that $z_0$ is non-wandering point of family $\mathcal F$. However the converse need not be true. For example, $0$ is non-wandering point for the family  $\mathcal{F}=\{f_n(z) = z^n: n\in\mathbb{N}\}$ and $0 \in F(\mathcal{F})$.

\medskip 

We denote by $\mathcal N_{\mathcal W}(\mathcal F)$, the set of all non-wandering points of family $\mathcal F$.

\begin{theorem}
Let $\mathcal F$ be a family of holomorphic functions from an open set $D$ to $\mathbb C$. Then $\mathcal N_{\mathcal W}(\mathcal F)$ is a closed set.
\end{theorem}
\begin{proof}
Let $z_0 \in \overline{\mathcal N_{\mathcal W}(\mathcal F)}$, and $U$ be an open neighbourhood of $z_0$. Then there exists a sequence $z_n \in \mathcal N_{\mathcal W}(\mathcal F)$ such that $z_n \rightarrow z_0$ as $n \rightarrow \infty$. Also $z_n \in \mathcal N_{\mathcal W}(\mathcal F)$ implies that for each open neighbourhood $U_n$ of $z_n$, there exists $f_n \in \mathcal F$ such that $f_n(U_n) \cap U_n \neq \emptyset$. Since $z_n \rightarrow z_0$ as $n \rightarrow \infty$, we can adjust the radius of $U_n$'s such that $U_n \subseteq U$ for large values of $n$. Therefore  $f_n(U) \subseteq f_n(U)$ so that $f_n(U_n)\cap U_n \subseteq f_n(U) \cap U$. This implies that $f_n(U) \cap U \neq \emptyset$, and thus $z_0 \in \mathcal N_{\mathcal W}(\mathcal F)$. Hence $\mathcal N_{\mathcal W}(\mathcal F)$ is closed.

\end{proof}


\section{topologically transitive Families of Functions}\label{sec:transitivityconditions}
In this section, we examine the sufficient conditions for transitivity in different families of functions. Additionally, we explore the circumstances under which transitivity can be transferred from one family of functions to another.
\begin{theorem}
  Suppose that $F(\mathcal{F})$ and $J\left(\mathcal{F}\right)$ are forward invariant under the family $\mathcal{F}$ of holomorphic functions on $D$. Then the restriction of $\mathcal{F}$ to $J\left(\mathcal{F}\right)$ is topologically transitive with respect to $J\left(\mathcal{F}\right).$
\end{theorem}
\begin{proof} If $J\left(\mathcal{F}\right)$ is empty, there is nothing to prove.
  Let $G$ and $V$ be an open sets in $J\left(\mathcal{F}\right).$  Then we can write $G=U\cap J\left(\mathcal{F}\right)$ for some open set $U$ in $\mathbb{C}.$  So the set $\mathbb{C}\setminus\cup_{f\in\mathcal{F}}f(U)$ contains at most two points. Forward invariance of $F(\mathcal{F})$ and $J\left(\mathcal{F}\right)$ implies that 
  $$J(\mathcal{F})\subset \cup_{f\in\mathcal{F}}f(G).$$ Since $V\subset J(\mathcal{F})\subset \cup_{f\in\mathcal{F}}f(G),$ there is some $f\in\mathcal{F}$ such that $V\cap f(G)\neq \emptyset.$
\end{proof}

\bigskip 

Following result is another characterization of topologically transitive family:
\begin{theorem}
A family $\mathcal F \subseteq C(D, \mathbb C)$ is topologically transitive if and only if for any open set $V \subseteq \mathbb C$, the set $\{z \in D : f(z) \in V, \mbox{ for some }f \in \mathcal F\}$ is dense in $D$.
\end{theorem}
\begin{proof}
Suppose that $\mathcal F$ is topologically transitive. For any open set $V \subseteq \mathbb C$, we can write the set $\{z \in D : f(z) \in V, \mbox{ for some }f \in \mathcal F\}$ as 
$\bigcup\limits_{f \in \mathcal F}f^{-1}(V) = W (\text{say}).$ Let $U$ be any open set in $D$. Since $\mathcal F$ is topologically transitive, there exists a function $f \in \mathcal F$ such that $f(U) \cap V \neq \emptyset$. Take a point $z_0 \in f(U) \cap V$. Then $z_0 = f(\zeta)$ for some $\zeta \in U$ and $z_0 \in V$,  so that $\zeta \in f^{-1}(V)$. This shows that $U \cap W \neq \emptyset$, and hence $W$ is dense in $D$.

\medskip

For the converse, consider a pair of open sets $U \subseteq D$ and $V\subseteq \mathbb C$. Then the set $\{z \in D: f(z) \in V \mbox{ for some } f\in\mathcal{F}\}$ intersects  $U$. Therefore, we can choose a point $z \in U$ such that $f(z) \in V$, and so $f(z) \in f(U) \cap V$. This shows that $f(U) \cap V \neq \emptyset$, and hence $\mathcal F$ is topologically transitive.
\end{proof}
Next, we consider expanding families of  meromorphic functions.
 A family $\mathcal{F}$ of meromorphic functions on $D$ is said to be  {\it expanding at $z_0 \in D$} with respect to a set $A \subset Y$ if for every open neighbourhood $U$ of $z_0$ and every compact set $K \subset A$ we have $K \subset f(U)$ for infinitely many $f \in \mathcal{F}$ (see \cite{meyrath2022}).

Our next result shows that if a family $\mathcal{F} $   of meromorphic functions on an open set  $D$ is expanding at point $z_0$ with respect to the complement of a set with empty interior, then the family is topologically transitive.
\begin{theorem}
Let $\mathcal{F} $ be a family   of meromorphic functions on an open set  $D$ in $\mathbb{C}$. If $\mathcal F$ is expanding at point $z_0$ in $D$ with respect to $\mathbb C_{\infty} \setminus E$, where $E^\circ = \emptyset$, then $\mathcal F$ is topologically transitive at $z_0$.
\end{theorem}
\begin{proof}
Suppose that $\mathcal F$ is expanding at point $z_0$ in $D$ with respect to $\mathbb C_{\infty} \setminus E$, where $E^\circ = \emptyset$. Then for each open neighbourhood $U$ of $z_0$ and a compact set $K \subset \mathbb C_\infty \setminus E$, $K \subset f(U)$ for infinitely many $f \in \mathcal F$. Since $E^\circ = \emptyset$, for each open set $V$ in $\mathbb C_\infty$ we have $V \cap \mathbb C_\infty \setminus E \neq \emptyset$. So, we can choose a compact set $K \subset V \cap \mathbb C_\infty \setminus E$, and therefore by above argument $K \subseteq f(U)$ for infinitely many $f \in \mathcal F$. Thus $f(U) \cap V \neq \emptyset$, and hence $\mathcal F$ is topologically transitive at $z_0$.
\end{proof}
\begin{lemma}\label{lem:compacttransitive} Let $\mathcal{F}\subseteq C(D,\mathbb{C})$ be a compact family. Suppose that $\mathcal{F}$ is topologically transitive at $z_0.$ Then for any open set $U$ in $D$ containing $z_0$ and an open set $V$ in $\mathbb{C},$ there exists $f\in\mathcal{F}$ such that $f(z_0)\in f(U)\cap V.$
\end{lemma}
\begin{proof}
Consider a pair of open sets $U\subseteq D$, $V\subseteq \mathbb{C}$ and $z_0\in U.$ Choose an open set $G$ such that ${\overline{G}}\subseteq V.$ Without loss of generality, we can assume that the open disks $D(z_0,1/n)$ are contained in $U$ for all $n.$ Since $\mathcal{F}$ is topologically transitive, we get a sequence $\{f_n\}$ in $\mathcal{F}$ and a sequence ${z_n}\in D(z_0, 1/n)$ such that $f_n(z_n)$ is in $ f_n(D(z_0,1/n))\cap G.$
Compactness of $\mathcal{F}$ implies that $\{f_n\}$ has subsequence, we rename it again $\{f_n\},$ which converges locally uniformly to some $f\in\mathcal{F}$ in $D.$ Since $z_n\to z_0$ as $n\to\infty,$ we see that $f_n(z_n)\to f(z_0)$ as $n\to\infty.$ So $f(z_0)\in{\overline{G}}$, since $f_n(z_n)\in G$ for all $n.$ Hence $f(z_0)\in f(U)\cap V.$
 \end{proof}
Next we investigate the various conditions under which transitivity can be borrowed from one family of functions to another.
 \begin{theorem}
     Let $\mathcal{F}\subseteq C(D,\mathbb{C})$ be a compact family and $\mathcal{G}\subseteq C(D,\mathbb{C})$ be such that for each $f\in\mathcal{F}$ there is $g\in\mathcal{G}$ satisfying $f(z_0)=g(z_0).$ Then transitivity of $\mathcal{F}$  at $z_0$ implies transitivity of $\mathcal{G}$ at $z_0.$  
 \end{theorem}
 \begin{proof}
    Consider a pair of open sets $U\subseteq D$, $V\subseteq \mathbb{C}$ and $z_0\in U.$ By Lemma \ref{lem:compacttransitive}, there is an $f\in\mathcal{F}$ such that $f(z_0)\in f(U)\cap V.$ Let $g\in\mathcal{G}$ be such that $f(z_0)=g(z_0).$ Then $g(z_0)\in g(U)\cap V.$ Hence $\mathcal{G}$ is topologically transitive.
 \end{proof}
 \begin{theorem}\label{Thm:continuityliketransitivity} Let $\mathcal{F}, \mathcal{G} \subseteq C(D,\mathbb{C})$ and suppose that $\mathcal{F}$ is topologically transitive on $D$ and that each $f \in \mathcal{F}$ is an open map. If for each $f \in \mathcal{F}$ there exists a $g \in \mathcal{G}$ such that for every $\epsilon > 0$ there exists a $\delta > 0$ satisfying $\left|f(z) - g(w)\right| < \epsilon$ whenever $|z - w| < \delta$, then $\mathcal{G}$ is topologically transitive on $D.$
\end{theorem}

\begin{proof} Let $U \subseteq D$ and $V \subseteq \mathbb C$ be a pair of open sets. By the transitivity of $\mathcal{F}$, there exists $f \in \mathcal{F}$ such that $f(U) \cap V \neq \emptyset$. Since $f(U)$ is an open set and $f$ is continuous, we can choose an open disk $D(f(z_0), \epsilon)$ in $f(U) \cap V$ such that there exists an open set $G \subseteq U$ containing $z_0$ with $f(G) \subseteq D(f(z_0), \epsilon)$. From the hypothesis, there exists $g \in \mathcal{G}$ and $\delta > 0$ such that $|f(z_0) - g(z)| < \epsilon$ whenever $|z_0 - z| < \delta$. This implies $g\left(D(z_0, \delta)\right) \subseteq D(f(z_0), \epsilon) \subseteq V$. Hence, $g(U) \cap V \neq \emptyset$. This shows that $\mathcal{G}$ is topologically transitive on $D.$ \end{proof}
\section{Alternate Proof of Montel's Theorem}\label{sec:montelstheorem}
Montel's Theorem is a pivotal result in complex analysis concerning normal families of functions.   In this section, we provide an alternate proof of Montel's theorem. 

 A family $\mathcal{F}$ of meromorphic functions is said to be\textit{ weakly mixing} with respect to $z_0$ if for any open neighbourhood $U$ of $z_0$ and a pair of open sets $V_1, V_2$ there is $f\in \mathcal{F}$ such that $f(U)\cap V_i\neq \emptyset$ (see \cite{meyrath2022} ).

T. Meyrath \cite{meyrath2022} in 2022 obtained the following criterion:
\begin{theorem}\label{meyrath}
   A family  $\mathcal{F}$ of meromorphic functions is non normal at $z_0$ if and only if $\mathcal{F}$ is weakly mixing with respect to $z_0$
\end{theorem}
By using Theorem \ref{meyrath}, we give an alternate proof of Montel's theorem as:
\begin{theorem}
Let $\mathcal{F}$ be a locally bounded family of holomorphic functions at $z_0,$ then $\mathcal{F}$ is normal at $z_0.$
\end{theorem}
\begin{proof}
If $\mathcal{F}$ is non normal, by Theorem \ref{meyrath} $\mathcal{F}$ is weakly mixing with respect to $z_0.$ Let $U$ be a neighbourhood of $z_0$ and $M>0$ be such that $|f(z)|\leq M$ for all $z\in U.$ Weakly mixing implies that there is $f\in \mathcal{F}$ such that $f(U)\cap V_i\neq \emptyset,$ where $V_1=V_2$ is a exterior of a disk centered at origin and radius $M,$ which is a contradiction.
\end{proof}

\medskip

T. Meyrath [\cite{meyrath2022}, Corollary 1] obtained the following characterization of non-normality in terms of the expanding property:

\begin{theorem} \label{mcriteria} Let $\mathcal{F} $ be a family   of meromorphic functions on an open set  $D$ in $\mathbb{C}$, and $z_0 \in D$. Then the following are equivalent:
\begin{enumerate}
    \item There exists $A \subset \mathbb{C}_{\infty}$ with $|A| \geq 2$ such that $\mathcal{F}$ is expanding at $z_0$ with respect to $A$.
    \item $\mathcal{F}$ is non-normal at $z_0$.
    \item There exists $E \subset \mathbb{C}_{\infty}$ with $|E| \leq 2$ such that $\mathcal{F}$ is expanding at $z_0$ with respect to $\mathbb{C}_{\infty} \setminus E$.
\end{enumerate}
\end{theorem}
Using Theorem \ref{mcriteria} we give an alternate proof of {\it Montel's three point theorem}:
\begin{theorem}
A family $ \mathcal{F}$ of meromorphic functions on a domain $D$ in $\mathbb C$ that omits three distinct values $a, b, c \in \mathbb C_{\infty}$ is normal on $D$.
\end{theorem}
\begin{proof} 
Assume that the family $ \mathcal{F}$ omits three distinct values $a, b, c \in \mathbb C_{\infty}$. Suppose, for the sake of contradiction, that $ \mathcal{F}$ is not normal at point $z_0 \in D$. By Theorem \ref{mcriteria}, there exists a set \ $E \subset \mathbb{C}_{\infty}$ with $|E| \leq 2$ such that $\mathcal{F}$ is expanding at $z_0$ with respect to $\mathbb{C}_{\infty} \setminus E$. Therefore for every 
open neighborhood $U$ of $z_0$ and every compact set $K \subseteq \mathbb{C}_{\infty} \setminus E$ we have 
\begin{equation}\label{m3t}
K \subset f(U)\;\text{for infinitely many}\;f \in \mathcal{F}.
\end{equation}
Since $D$ is open, we can choose a disk $D(z_0, r)$ such that $\overline{D}(z_0, r) \subset D$. Additionally, since $|E| \leq 2$, at least one of the three values $a, b, c$ (let's say $a$) belongs to $\mathbb{C}_{\infty} \setminus E$. If we take $U = D(z_0, r/2)$ and $K = \{a\}$, then Theorem (\ref{m3t}) implies $K \subset f(U)$ for infinitely many $f \in \mathcal{F}$ so that $a \in f(U)$ for infinitely many $f \in \mathcal{F}$. However, this implies that the value 
$a$ is assumed by the family 
$\mathcal F$ on $D$, resulting in a contradiction. Hence the result follows.
\end{proof}


\end{document}